\documentclass[letterpaper,10pt]{article}

\usepackage[affil-it]{authblk}
\usepackage{amsthm,latexsym,amsfonts,amsmath,amssymb,amsthm}
\usepackage[round,authoryear]{natbib}
\usepackage{mathrsfs} 
\usepackage[pdftex,colorlinks=true,linkcolor=blue,unicode]{hyperref} 
\usepackage[margin=1.5in]{geometry} 

\newcommand{\R}{\mathbb{R}}

\newcommand{\Z}{\mathbb{Z}}
\newcommand{\EE}{\mathbb{E}}
\newcommand{\PP}{\mathbb{P}}

\newcommand{\CC}{\mathscr{C}}

\newcommand{\dd}{\mathrm{d}}
\newcommand{\vvv}{v} 
\newcommand{\argmin}{\mathop{\mathrm{arg \, min}}}
\newcommand{\arginf}{\mathop{\mathrm{arg \, inf}}}
\newcommand{\hit}{\mathrm{hit}}
\newcommand{\ZZ}[2]{Z^{(#1)}_{#2}}
\newcommand{\bb}[2]{b^{(#1)}_{#2}}
\newcommand{\UU}[2]{U^{(#1)}_{#2}}
\newcommand{\m}{m}
\newcommand{\lo}{l}
\newcommand{\nK}[1]{\| #1 \|_\mathsf{K}}

\newcommand{\A}{A}
\newcommand{\ttt}{\xi}
	
\newtheorem{theorem}{Theorem}
\newtheorem{lemma}{Lemma}
\newtheorem{proposition}{Proposition}
\newtheorem{corollary}{Corollary}
\newtheorem{definition}{Definition}

\title{On the Euler discretization error of Brownian motion about random times}

\author{A.~B.~Dieker\thanks{Postal address: Department of Industrial Engineering and Operations Research, Columbia University, New York, NY, 10027, USA}}
\affil{Columbia University}
\author{Guido R.~Lagos\thanks{Postal address: Faculty of Engineering and Sciences, Universidad Adolfo Ib\'a\~nez, Pe\~nalol\'en, Santiago, Chile}}
\affil{Universidad Adolfo Ib\'a\~nez}
\date{\today}

\begin{document}

\maketitle

\begin{abstract}
In this paper we derive weak limits for the discretization errors of sampling barrier-hitting and extreme events of Brownian motion by using the Euler discretization simulation method.
Specifically, we consider the Euler discretization approximation of Brownian motion to sample barrier-hitting events, i.e.~hitting for the first time a deterministic ``barrier'' function; and to sample extreme events, i.e.~attaining a minimum on a given compact time interval or unbounded closed time interval.
For each case we study the discretization error between the actual time the event occurs versus the time the event occurs for the discretized path, and also the discretization error on the position of the Brownian motion at these times.
We show limits in distribution for the discretization errors normalized by their convergence rate, and give closed-form analytic expressions for the limiting random variables.
Additionally, we use these limits to study the asymptotic behaviour of Gaussian random walks in the following situations: (1.) the overshoot of a Gaussian walk above a barrier that goes to infinity; (2.)~the minimum of a Gaussian walk compared to the minimum of the Brownian motion obtained when interpolating the Gaussian walk with Brownian bridges, both up to the same time horizon that goes to infinity; and (3.)~the global minimum of a Gaussian walk compared to the global minimum of the Brownian motion obtained when interpolating the Gaussian walk with Brownian bridges, when both have the same positive drift decreasing to zero.
In deriving these limits in distribution we provide a unified framework to understand the relation between several papers where the constant $-\zeta(1/2)/\sqrt{2 \pi}$ has appeared, where $\zeta$ is the Riemann zeta function.
In particular, we show that this constant is the mean of some of the limiting distributions we derive.
\end{abstract}

\section{Introduction}

Brownian motion is arguably the most important continuous-time stochastic process in probability theory.
Its relevance is only increased by its widespread use as a stochastic model in engineering, sciences and business.
On the other hand, the simulation of Brownian motion raises fundamental challenges, as it is a continuous-time process characterized by its rapid movement and self-similar-type structure in time-space, whereas computers can only simulate and store discrete objects.
This means that the simulation of Brownian motion is inherently inaccurate, except for a few especially structured collection of events where Brownian motion can be simulated without bias, see e.g.~\cite{devroye2010exact} for a survey on such special cases.

In this paper we consider the simulation of Brownian motion by approximating it on a constant, regularly spaced, time mesh; that is, we consider the {\it Euler discretization} of Brownian paths on a {\it regular} mesh.
This approximation is arguably the easiest and simplest possible simulation method for Brownian motion, which makes it appealing from a practical point of view.
Moreover, in some special practical applications there are exogenous conditions which makes the Euler discretization the most sensible simulation method to use; for example, in finance when a financial instrument can only be monitored at regular time intervals, see~\cite{broadie1997continuity}.

We are particularly interested in the simulation of {\it extreme events} and of {\it barrier-hitting events} of Brownian motion.
We call a {\it barrier-hitting event} an event where the Brownian motion ``hits'' or ``crosses'' for the first time a given ``barrier'', possibly time-dependent.
We call {\it extreme event} an event where a minimum (or maximum) of the Brownian motion over a closed time interval is attained; in this case the time interval can be bounded or unbounded.
These two types of events are of fundamental importance, both in the theory of stochastic processes, see e.g.~fluctuation theory, as well as in practice, where usually events of critical interest can be formulated as one of these events, e.g.~a stock price ever reaching a certain value, or a natural disaster occurring in a certain time horizon.

In this paper we study the accuracy of simulating barrier-hitting and extreme events of Brownian motion by instead simulating the respective event for the Euler discretized Brownian motion.
The accuracy of the Euler discretization is a non-trivial issue, since, except for trivial cases, with probability one none of these two events will ever occur exactly on a regular time mesh of the form $\{ 0, T/n, 2T/n, \ldots \}$ for constant $T$.
This raises the question of what theoretical conditions are there that guarantee accuracy of the Euler discretization.
This question motivates the main objective of this paper, which is: the study of convergence of the discretization error for these two events, as well as their rate of convergence, and weak convergence of the normalized errors.


Our work is motivated by a desire to understand an elusive connection between three important objects: (i) discretized Brownian motion, (ii) Gaussian random walks, i.e., random walks with iid normal step sizes, and (iii) the constant $-\zeta(1/2) / \sqrt{2 \pi}$, where $\zeta$ is the Riemann zeta function.
Indeed, this latter constant has appeared in a number of papers working with Brownian motion or Gaussian walks, see e.g.~\cite{chernoff1965sequential,siegmund1979corrected,chang1997ladder,janssen2007lerch,calvin1995average,asmussen1995discretization,broadie1997continuity}.
As~\cite{janssen2007lerch} point out, the connection between these works containing this particular constant has been unclear.
Moreover, the connection is especially intriguing since these works are about different topics in a variety of fields, e.g.~\cite{chernoff1965sequential} studies a problem of sequential analysis in statistics and~\cite{broadie1997continuity} studies option pricing in financial engineering.
Our work is largely motivated by elucidating the precise connection between the papers in the literature containing the constant $-\zeta(1/2) / \sqrt{2 \pi}$.
This paper explains how these works are connected.

\paragraph{Main contributions}
We consider barrier-hitting events of a Brownian motion $B$, where the barrier function is continuously differentiable and non-decreasing on $\R_+$, and which has initial position at time zero above the Brownian motion.
We also consider the extreme events where the Brownian motion attains a global minimum over $[0, \infty)$, or attains its minimum over an interval of the type $[0,a]$ for a given $a>0$ finite.
In any of these three cases denote for now, if they are finite, by $t^*$ the time at which the event occurs and $B(t^*)$ the position of the Brownian motion at such time; denote analogously by $t^n$ and $B^n(t^n)$ the respective values for the Euler discretization approximation $B^n$ of $B$ on the mesh $\{0, 1/n, 2/n, \ldots \}$.

A first contribution is that we show the following convergence results conditional on $t^*$ and $B(t^*)$ being finite.
We show that the normalized errors of time and position of the Brownian motion, i.e.~$n \left( t^n - t^* \right)$ and $\sqrt n \left( B^n(t^n) - B(t^*) \right)$, converge jointly in distribution and we give a closed-form analytic expression for the limiting random variables.
In the case of barrier-hitting events, the limiting random variable is related to the overshoot above zero of an ``equilibrium'' Gaussian random walk.
In the case of extreme events, the joint limiting random variable is the same in both cases, and involves the minimum of a two-sided Bessel process over the integers displaced by a uniform random variable.
Nonetheless, we remark that some of the derived limits in distribution were at least partially known in the literature.
Specifically, our Theorem~\ref{theo1} contains \cite[Lemma~10.10]{siegmund1985sequential} in the Gaussian walk case and also contains \cite[Lemma~4.2]{broadie1997continuity}; and our Theorem~\ref{theo2} contains \cite[Theorem~1]{asmussen1995discretization}.
Our work contributes in augmenting these results and also in giving a unified derivation of them.

A second contribution of this paper is that our Euler discretization limits in distribution allow to give rise to other limits in distribution for Gaussian random walks.
These limits are related to (1.) the overshoot of the Gaussian walk above a barrier that goes to infinity; (2.) the minimum of the Gaussian walk compared to the minimum of a Brownian motion, both up to the same time horizon that goes to infinity; and (3.) the global minimum of the Gaussian walk compared to the global minimum of a Brownian motion, when both have the same positive drift decreasing to zero.
The limiting distributions of these quantities are some the same limiting random variables obtained for the Euler discretization errors, since the limits are derived as simple corollaries of the Euler discretization analysis.
In particular, this allows us to give a crisper intuition of the limiting distributions obtained for the Euler discretization errors, since from the Gaussian walk perspective the limiting random variables can be understood as the equilibrium distributions of renewal processes.

A third contribution of this work, and perhaps the most important, is that we give a clean connection between several papers in the literature where the constant $-\zeta(1/2) / \sqrt{2 \pi}$ has appeared, where $\zeta$ is the Riemann zeta function.
This constant has appeared in works about Gaussian walks, \cite{siegmund1979corrected,chang1997ladder,janssen2007lerch,janssen2007cumulants}; discretized Brownian motion, \cite{calvin1995average,asmussen1995discretization,broadie1997continuity,broadie1999connecting}; and approximations of stochastic processes, \cite{chernoff1965sequential,siegmund1979corrected,comtet2005precise}.
We show that the constant $-\zeta(1/2) / \sqrt{2 \pi}$ is the mean of two of the limit random variables we derive; and moreover, either one of these two limit distributions appear in each of our six results, i.e.,~in Theorems~\ref{theo1}, \ref{theo2} and~\ref{theo3} and Corollaries~\ref{corol1}, \ref{corol2} and~\ref{corol3}.
On the other hand, we argue that in all the papers containing the constant $-\zeta(1/2) / \sqrt{2 \pi}$ the part where this term appears can be understood as extensions, particular cases or applications of one of our six results.
In this way our work allows to view these literature from a unified perspective, clarifying the relationship between them.





\paragraph{Literature review}

A variety of methods for simulation of Brownian motion exist, given its popularity and usefulness.
Sophisticated methods include {\it exact simulation} of some special quantities, see \cite{devroye2010exact}, and the simulation of approximations of Brownian motion that have path-wise guarantees of accuracy, see~\cite{beskos2012varepsilon}.
On the other hand, the Euler discretization approximation is a classic and simple approach to sample Brownian motion and more general diffusions, see e.g.~\cite{platen1999introduction} for a comprehensive exposition.
Nonetheless, the first approach along the lines of our work is~\cite{asmussen1995discretization}, who study the error of the Euler discretization approach to sample one-dimensional reflected Brownian motion.
We highlight the work of~\cite{broadie1997continuity,broadie1999connecting}, which uses the Euler discretization approximation to study the pricing of barrier and lookback options, respectively.
They justify their choice of the discretized approximation in that the financial instruments they study can only be monitored at a pre-specified regular time mesh and not in continuous-time.
\cite{calvin1997average} studies the discretization error of time and position of the extremes of Brownian motion over a compact time interval, but for the case when the discretization points are drawn from a Poisson point process.
\cite{bisewski2017controlling} consider non-equidistant discretization of Brownian motion to reduce the bias in the estimation of barrier-crossing probabilities.
In greater generality, the last decade has seen the development of several sophisticated methods for exact sampling of more general diffusion processes, see e.g.~\cite{beskos2005exact,beskos2006retrospective}.
In this line, a work related to ours is~\cite{etore2013exact}, who study exact simulation of diffusions that involve the local-time at zero; nonetheless, their work cannot be extended to ours since they do not treat reflected diffusions.

A central element of our work is the convergence in distribution of the Brownian motion when ``zoomed-in'' in time-space about random times.
The study of this convergence heavily relies on path decompositions of Brownian motion; \cite{williams1974path} is the quintessential reference in this line.
For our particular study we use decompositions of Brownian motion about its global minima as Bessel process, see~\cite{rogers1981markov}, \cite{bertoin1991decomposition}; and also decomposition about its local minima, see~\cite{asmussen1995discretization} and~\cite{imhof1984density}.
A related work is~\cite{chaumont1996conditionings}, who studies the pre- and post-minimum paths of L\'evy processes, but when process is conditioned to stay positive.
In our case we do not condition on the process staying positive.

In several works in the literature the constant $-\zeta(1/2) / \sqrt{2 \pi}$ has appeared, where $\zeta$ is the Riemann zeta function.
These works essentially deal with extreme events of discretized Brownian motion, e.g.~\cite{asmussen1995discretization}, \cite{calvin1995average}, \cite{broadie1997continuity,broadie1999connecting}; with Gaussian random walks, e.g.~\cite{janssen2007lerch}, \cite{chang1997ladder}; and with approximations of stochastic processes, e.g.~\cite{chernoff1965sequential}, \cite{siegmund1979corrected}, \cite{comtet2005precise}.
More generally, \cite{biane2001probability} gives an overview of multiple random variables related to Brownian motion where the Riemann zeta and Jacobi's theta functions appear.

\paragraph{Outline}
In Section~\ref{ch3:sec:main results} we show the main results of this paper concerning the Euler discretization error of Brownian motion about random times of interest.
In detail, in Section~\ref{ch3:subsec:hit} we treat the error for barrier-hitting events; in Section~\ref{ch3:subsec:min a} the error for extreme events on compact time intervals, and in Section~\ref{ch3:subsec:min infty} we treat the error for extreme events on unbounded closed time intervals.

In Section~\ref{ch3:sec:gaussian} we transform the convergence in distribution of the normalized Euler errors into asymptotic results for Gaussian walks.
More precisely, in Section~\ref{ch3:sec:gaussian 1} we analyze the overshoot above a barrier when the barrier goes to infinity; in Section~\ref{ch3:sec:gaussian 2} we study the minimum of a Gaussian walk compared to the minimum of a Brownian motion, both up to the same time horizon that goes to infinity; and in Section~\ref{ch3:sec:gaussian 3} we analyze the global minimum of the Gaussian walk compared to the global minimum of a Brownian motion, when both have the same positive drift decreasing to zero.

In Section~\ref{sec:zeta} we show how our results provide a framework unifying several works in the literature where the constant $-\zeta(1/2) / \sqrt{2 \pi}$ has appeared, where $\zeta$ is the Riemann zeta function.

Finally, Section~\ref{ch3:sec:proof} is dedicated to the proofs of the main results in Section~\ref{ch3:sec:main results}.
For that, in Section~\ref{ch3:subsec:errors as maps} we write the discretization errors as mappings of Brownian ``zoomed-in'' about the random time of interest; then in Section~\ref{ch3:subsec:zoomed-in converge K} we prove that such zoomed-in processes converge in distribution; and in Section~\ref{ch3:subsec:proof} we give the actual proofs of the main results of this paper, Theorems~\ref{theo1}, \ref{theo2} and~\ref{theo3} of Section~\ref{ch3:sec:main results}.

\paragraph{Notation}
We denote as $\R_+$ and $\Z_+$ the nonnegative real numbers and integer numbers, respectively.
The function $\lceil x \rceil$ denotes the integer part of the number $x$.
Also, $\CC(\R)$ and $\CC(\R_+)$ are the set of real functions defined on all real numbers and on the nonnegative real numbers, respectively.
We call {\it standard Brownian motion} and {\it standard Bessel(3) process} as the Brownian motion and Bessel(3) process, respectively, with no drift and unit variance.
Unless otherwise stated, we assume that all process start from zero at time zero.

\section{Main results}\label{ch3:sec:main results}

In this section we show our main results about the error convergence when simulating barrier hitting and extreme value events of Brownian motion by using an Euler discretization approach.
In Section~\ref{ch3:subsec:hit} we focus on the error of estimating barrier hitting events with the Euler discretization.
In Section~\ref{ch3:subsec:min a} we consider the error when sampling extreme values of the Brownian motion over a fixed finite horizon.
In Section~\ref{ch3:subsec:min infty} we extend the analysis to extreme values over an infinite horizon.

Consider a Brownian motion $B = \left( B(t) : t \in \R_+ \right)$ with constant drift $\mu$ and variance $\sigma^2$.
For $n > 0$ integer, let $B^n = \left( B^n (t) : t \in \R_+ \right)$ be the piecewise constant process defined as
\begin{eqnarray}\label{ch3:def:Bn}
B^n (t) &:=& B \left( \lfloor n t \rfloor / n \right) \quad \text{for all} \quad t \in \R_+;
\end{eqnarray}
 that is, $B^n$ is the Euler discretization of $B$ on the mesh $\left\{ 0, \ 1/n, \ 2/n, \ \ldots \right\}$.

\subsection{Error of barrier hitting times}\label{ch3:subsec:hit}

We consider first the error of estimating barrier hitting times using the Euler discretization of the Brownian motion.
Consider a deterministic barrier function $b = (b(t) : t \geq 0)$ satisfying the following assumptions:
\begin{itemize}
	\item[(H$_b$)] The function $b: \R \to \R$ is continuous and nondecreasing on $\R_+$, continuously differentiable on $\R_+ \!\! \setminus \!\! \{ 0 \}$, and $b(0) \geq 0$. 
\end{itemize}
We want to estimate the time
\begin{eqnarray}\label{ch3:def:tau b}
\tau_b &:=& \inf \left\{ t \in (0,\infty) : B(t) \geq b(t) \right\};
\end{eqnarray}
for which we use the approximation
\begin{eqnarray}\label{ch3:def:tau b n}
\tau_b^n &:=& \inf \left\{ t \in (0,\infty) : B^n(t) \geq b(t) \right\},
\end{eqnarray}
where $B^n$ is the discretized version of $B$ defined in \eqref{ch3:def:Bn}.
Here we use the usual convention that $\inf \emptyset = +\infty$.
Our objective is to study the error of $\tau_b^n - \tau_b$ and $B^n(\tau_b^n) - B(\tau_b)$.
The following result establishes the convergence rate and limiting distribution of these errors.

\begin{theorem}\label{theo1}
Consider a barrier function $b = (b(t) : t \geq 0)$ satisfying (H$_b$).
Conditioned on the event $\left\{ \tau_b < \infty \right\}$, as $n \to \infty$ the triplet
\begin{eqnarray}\label{ch3:theo:hit b}
& \left( n \left( \tau_b^n - \tau_b \right) , \quad \sqrt n \left( B^n \left( \tau_b^n \right) - B \left( \tau_b \right) \right) , \quad \lceil n \tau_b \rceil - n \tau_b \right),
\end{eqnarray}
converges jointly in distribution to the triplet
\begin{eqnarray}\label{ch3:theo:lim hit b}
& \left( U+\min\{k \geq 0 : W(U+k)>0\} , \quad \sigma W \left( U+\min\{k \geq 0 : W(U+k)>0\} \right) , \quad U \right),
\end{eqnarray}
where $W = \left( W(t) : t \in \R_+ \right)$ is a standard Brownian motion independent of $B$, and $U$ is a uniformly distributed random variable on $(0,1)$ which is independent of $B$ and $W$.
\end{theorem}

We remark that in Section~\ref{ch3:sec:gaussian 1} we give another limit random variable that has the same distribution of~\eqref{ch3:theo:lim hit b}.
Moreover, this characterization connects the random variable~\eqref{ch3:theo:lim hit b} to the distribution of the {\it ladder heights} of a Gaussian random walk; see Section~\ref{ch3:sec:gaussian 1} for further details.

\subsection{Error of minimum on finite horizon}\label{ch3:subsec:min a}

We consider now the error of estimating the minimum on bounded time intervals by using the Euler discretization of the Brownian motion.
Consider a time interval $[0, a]$ such that $0 < a < \infty$.
We want to estimate the time
\begin{eqnarray}\label{ch3:def:T min fin}
T_{\min,a} &:=& \inf \left\{ t \in [0,a] : B (t) = \min_{s \in [0,a]} B (s) \right\};
\end{eqnarray}
for which we use the approximation
\begin{eqnarray}\label{ch3:def:T min fin n}
T_{\min,a}^n &:=& \inf \left\{ t \in [0,a] : B^n (t) = \min_{s \in [0,a]} B^n (s) \right\};
\end{eqnarray}
where the discretized version $B^n$ is defined in \eqref{ch3:def:Bn}.
Here $T_{\min,a}$ is the almost surely unique value satisfying $B (T_{\min,a}) = \min_{s \in [0,a]} B (s)$.
Nonetheless, for all $t$ in $[ T_{\min,a}^n, \ T_{\min,a}^n+1/n )$ it holds that $B^n (t) = \min_{s \in [0,a]} B^n (s)$ because $B^n$ is defined in~\eqref{ch3:def:Bn} as a \emph{piecewise constant} interpolation of $B$.

We want to study the error of $T_{\min,a}^n - T_{\min,a}$ and $B^n(T_{\min,a}^n) - B(T_{\min,a})$.
The following result establishes the limiting distribution and convergence rate of these errors.

\begin{theorem}\label{theo2}
The triplet
\begin{eqnarray}\label{ch3:theo:min a}
& \left( n \left( T_{\min,a}^n - T_{\min,a} \right) , \quad \sqrt n \left( B^n \left( T_{\min,a}^n \right) - B \left( T_{\min,a} \right) \right) , \quad \lceil n T_{\min,a} \rceil - n T_{\min,a} \right),
\end{eqnarray}
converges jointly in distribution to the triplet
\begin{eqnarray}\label{ch3:theo:lim min a}
& \left( U+\argmin_{k \in \Z} R (U+k) , \quad \sigma \min_{k \in \Z} R (U+k) , \quad U \right),
\end{eqnarray}
where $R  = \left( R (t) : t \in \R \right)$ is a two-sided Bessel(3) process and $U$ is a uniformly distributed random variable on $(0,1)$ which is independent of $R$.
Here we have abused notation by denoting as $\argmin$ in \eqref{ch3:theo:lim min a} the almost surely unique value $k$ at which the minimum is attained.
\end{theorem}

We remark that the convergence of the second component of the triplet in~\eqref{ch3:theo:min a} corresponds to~\cite[Theorem~1]{asmussen1995discretization}.
Their limit however contains an extra time term which in our framework would correspond to $\sqrt a$; this term appears because they consider a discretization with step size $a/n$, whereas we consider $1/n$.
The objective of~\cite{asmussen1995discretization} is to study the Euler discretization error $\Gamma B(t) - \Gamma B^n(t)$ of the {\it reflected Brownian motion} $\Gamma B$.
Here, $\Gamma$ is the {\it reflection mapping}, defined for all trajectory $X$ with $X(0)=0$ as
$$\Gamma X (t) := X(t) - \inf_{s \in [0,t]} X(s),$$
Our work extends theirs in the following way.
Let the {\it busy period} mapping $\Gamma'$ be
$$\Gamma' X (t) := t - \sup \left\{ s \in [0,t] : \Gamma X(s) = 0 \right\} = t - \sup \left\{ s \in [0,t] : X(s) = \inf_{u \in [0,s]} X(u) \right\}$$
for $X$ trajectory with $X(0)=0$.
Alternatively, $\Gamma'$ can be seen as the {\it drift derivative} of the reflection mapping; see \cite{dieker2014sensitivity} for further details on the latter process.
It is easy to show that if $B^n$ is the Euler discretization with stepsize $1/n$ as defined in \eqref{ch3:def:Bn} then
\begin{eqnarray*}
& n \left( \Gamma' B(t) - \Gamma' B^n(t) \right) = 1+n \left( T_{\min,t}^n - T_{\min,t} \right)
\end{eqnarray*}
and
\begin{eqnarray*}
& \sqrt n \left( \Gamma B(t) - \Gamma B^n(t) \right) = \sqrt n \left( B^n \left( T_{\min,t}^n \right) - B \left( T_{\min,t} \right) \right).
\end{eqnarray*}
In particular, Theorem~\ref{theo2} gives the joint limit in distribution of these Euler discretization errors.

Another related work is \cite{calvin1997average}, which also studies the discretization error of time and position of the extremes of Brownian motion over a compact time interval.
Nonetheless, his work is different from ours in that he considers the case when the discretization points are drawn from a Poisson point process.

\subsection{Error of minimum on an infinite horizon}\label{ch3:subsec:min infty}

We consider now the error of estimating the minimum on unbounded time intervals by using the Euler discretization of the Brownian motion.
We want to estimate the time
\begin{eqnarray}\label{ch3:def:T min infty}
T_{\text{min},\infty} &:=& \inf \left\{ t \in [0,\infty) : B (t) = \min_{u \in [0,\infty)} B (u) \right\};
\end{eqnarray}
which we approximate as follows by using the discretized path $B^n$ defined in \eqref{ch3:def:Bn}:
\begin{eqnarray}\label{ch3:def:T min inf n}
T_{\text{min},\infty}^n &:=& \inf \left\{ t \in [0,\infty) : B^n (t) = \min_{u \in [0,\infty)} B^n (u) \right\}.
\end{eqnarray}
For that, we study the errors $T_{\text{min},\infty}^n - T_{\text{min},\infty}$ and $B^n(T_{\text{min},\infty}^n) - B(T_{\text{min},\infty})$.
The following result establishes their convergence rate and limiting distribution.

\begin{theorem}\label{theo3}
As $n \to \infty$, the triplet
\begin{eqnarray}\label{ch3:theo:min infty}
& \left( n \left( T_{\text{min},\infty}^n - T_{\text{min},\infty} \right) , \quad \sqrt n \left( B^n(T_{\text{min},\infty}^n) - B(T_{\text{min},\infty}) \right) , \quad \lceil n T_{\text{min},\infty} \rceil \! - \! n T_{\text{min},\infty} \right),
\end{eqnarray}
converges jointly in distribution to the same triplet in \eqref{ch3:theo:lim min a}.
\end{theorem}

We remark that we are {\it not} claiming that both triplets \eqref{ch3:theo:min a} and \eqref{ch3:theo:min infty} converge {\it jointly} together to the same limit~\eqref{ch3:theo:lim min a}.

\section{Extension to Gaussian walks}\label{ch3:sec:gaussian}
 
In this section we extend our weak convergence error results to the setting of Gaussian walks.
A {\it Gaussian walk} is a discrete time stochastic process $S = \left( S_k : k \geq 0 \right)$ where $S_k := \sum_{i=1}^k X_i$ and the increments $X_i$ are iid normal random variables.
We will see that several important phenomena of these processes can be described using the limiting distributions found in the Theorems~\ref{theo1}, \ref{theo2} and \ref{theo3}.

\subsection{Corollary 1: limiting overshoot for increasing barrier}\label{ch3:sec:gaussian 1}

Consider a Gaussian walk $S = \left( S_n : n \geq 0 \right)$ starting at zero, with nonnegative drift $\EE S_1 = \nu \geq 0$ and with variance $\EE (S_1-\nu)^2 = \sigma^2$.
Assume that the probability space is sufficiently rich so that there exists a Brownian motion $B = \left( B(t) : t \geq 0 \right)$ starting from zero, with zero drift and variance $\sigma^2$, and such that almost surely $S_n = B(n)$ for all $n$.

We want to analyze ``how well'' a barrier-hitting event of the Brownian motion $B$ approximates the analogue event of the Gaussian walk $S$.
More specifically, defining for any fixed ``barrier'' $m>0$
$$\tau_m^S := \min \left\{ k \geq 0 : S_k \geq m \right\} \qquad \text{and}\qquad \tau_m^B := \min \left\{ t\geq 0 : B(t) \geq m \right\},$$
we want to study the errors $S_{\tau_m^S} - B \left( \tau_m^B \right)$ and $\tau_m^S - \tau_m^B$ as $m$ increases.

\begin{corollary}\label{corol1}
As $m \to \infty$ the error pair
\begin{eqnarray}\label{theo:cor 1 exp}
& \left(
\tau_m^S - \tau_m^B , \quad
S_{\tau_m^S} - B \left( \tau_m^B \right)
\right)\end{eqnarray}
converges in distribution to the pair
\begin{eqnarray}\label{theo:cor 1 lim}
& \left(
U+\tau_+, \quad
\sigma S^*_{\tau_+}
\right).
\end{eqnarray}
Here, $S^* = \left( S^*_k : k \geq 0 \right)$ is a modified Gaussian walk, defined as $S^*_0 := \sqrt U X^*_0$ and $S^*_{k+1} := S^*_k + X^*_{k+1}$ for $k \geq 0$, where $(X^*_k)_{k \geq 0}$ are iid standard normal random variables and $U$ is a uniform random variable on $(0,1)$ independent of $(X^*_k)_{k \geq 0}$.
Also,
$$\tau_+ = \tau_+^{S^*} := \min \left\{ k \geq 0 : S^*_k > 0 \right\},$$
that is, $\tau_+$ is the first strictly positive time of $S^*$.
\end{corollary}

The proof is a straightforward corollary of Theorem~\ref{theo1}.
Indeed, using Brownian scaling the error pair~\eqref{theo:cor 1 exp} is equal in distribution to the first and second component of~\eqref{ch3:theo:hit b} for $\mu=\nu$, barrier $b=1$ and $n = m^2$.

We remark that since $B \left( \tau_m^B \right)=m$ the latter result covers in particular the overshoot of the Gaussian walk $S$ above the barrier $m$, i.e.~$S_{\tau_m^S} - m$, as $m \to \infty$.
Such limit overshoot was already known to converge, and its limit distribution was characterized in terms of the {\it ladder heights} distribution of the random walk, see e.g.~\cite{siegmund1985sequential}.
Thus, our result connects the ladder heights of a Gaussian walk with the distribution of the pair~\eqref{theo:cor 1 lim}.

We also remark that the limit random variables~\eqref{theo:cor 1 lim} and~\eqref{ch3:theo:lim hit b} are related in that the triplet
$$\left( U+\tau_+ , \quad \sigma S^*_{\tau_+} , \quad U \right)$$
is equal in distribution to the limiting random variable~\eqref{ch3:theo:lim hit b}.

\subsection{Corollary 2: running minimum asymptotics}\label{ch3:sec:gaussian 2}

Let $S = \left( S_n : n \geq 0 \right)$ be a Gaussian walk starting from zero, having drift $\EE S_1 = 0$ and variance $\EE (S_1)^2 =\sigma^2$.
Assume that the probability space is sufficiently rich so that there exists a Brownian motion $B = \left( B(t) : t \geq 0 \right)$ starting from zero, with zero drift and variance $\sigma^2$, and such that almost surely $S_n = B(n)$ for all $n$.

We want to analyze ``how well'' a finite-horizon extreme event of the Brownian motion $B$ approximates the analogue event of the Gaussian walk $S$.
More specifically, we want to study the random variables $\argmin_{k=0,\ldots, n} S_k$ and $\min_{k=0,\ldots, n} S_k$ by comparing them to their Brownian counterparts $\argmin_{t \in [0,n]} B(t)$ and $\min_{t \in [0,n]} B(t)$, respectively, where in both cases we refer to $\argmin$ as the almost sure unique values at which the minimums are attained.
The following result characterises the convergence of the latter two errors.

\begin{corollary}\label{corol2}
As $n \to \infty$, the pair
\begin{eqnarray}\label{ch3:theo:cor 2 exp}
& \left(
\argmin_{k=0,\ldots, n} S_k - \argmin_{t \in [0,n]} B(t) , \quad
\min_{k=0,\ldots, n} S_k - \min_{t \in [0,n]} B(t)
\right)\end{eqnarray}
converges in distribution to the pair
\begin{eqnarray}\label{ch3:theo:cor 2 lim}
& \left(
U+\argmin_{k \in \Z} R(U+k) , \quad
\sigma \min_{k \in \Z} R(U+k)
\right),
\end{eqnarray}
where $R = \left( R(t) : t \in \R \right)$ is a two-sided Bessel(3) process and $U$ a uniformly distributed random variable on $(0,1)$ that is independent of $R$.
Here we have abused notation and in all cases the $\argmin$ operations corresponds to the almost surely {\it unique} value at which the minimum is attained in each case.
\end{corollary}

The proof is straightforward when noting that, using Brownian scaling, the pair \eqref{ch3:theo:cor 2 exp} is equal in distribution to the first two terms of the triplet in \eqref{ch3:theo:min a} of Theorem~\ref{theo2}.

\subsection{Corollary 3: minimum as drift vanishes}\label{ch3:sec:gaussian 3}

Let $S = \left( S_n : n \geq 0 \right)$ be a Gaussian walk starting from zero, with strictly positive drift $\EE S_1 = \nu > 0$ and variance $\EE (S_1-\nu)^2 =\sigma^2$.
Assume that the probability space is sufficiently rich so that there exists a Brownian motion $B = \left( B(t) : t \geq 0 \right)$ starting from zero, having drift $\nu>0$ and variance $\sigma^2$, and such that $S_n = B(n)$ for all $n \in \Z_+$.

We want to analyze ``how well'' an infinite-horizon extreme event of the Brownian motion $B$ approximates the analogue event of the Gaussian walk $S$.
For that we study the values $\argmin_{k \in \Z_+} S_k$ and $\min_{k \in \Z_+} S_k$ as the drift $\nu$ decreases to zero, and we do it by comparing them to their Brownian counterparts $\argmin_{t \in \R_+} B(t)$ and $\min_{t \in \R_+} B(t)$, respectively, where in both cases we refer to $\argmin$ as the almost sure unique values at which the minimums are attained.

\begin{corollary}\label{corol3}
As the drift decreases to zero, $\nu \searrow 0$, the pair
\begin{eqnarray}\label{ch3:theo:cor 3 exp}
& \left(
\argmin_{k \in \Z_+} S_k - \argmin_{t \in \R_+} B(t), \quad
\min_{k \in \Z_+} S_k - \min_{t \in \R_+} B(t)
\right)
\end{eqnarray}
converges in distribution to the pair in \eqref{ch3:theo:cor 2 lim}.
Here we have abused notation and in all cases the $\argmin$ operations corresponds to the almost surely {\it unique} value at which the minimum is attained in each case.
\end{corollary}

The proof is direct from Theorem~\ref{theo3}, since using Brownian scaling we obtain that \eqref{ch3:theo:cor 3 exp} is equal in distribution to the first two terms of the triplet in \eqref{ch3:theo:min infty} with drift $\mu=1$ and $n=\nu^{-2}$.

\section{Connecting literature with the constant $-\zeta(1/2) / \sqrt{2 \pi}$}\label{sec:zeta}

In this section we show how our results provide a framework unifying several works in the literature where the constant
\begin{eqnarray}\label{def:beta}
\beta &:=& - \frac{ \zeta(1/2) }{ \sqrt{2 \pi} }
\end{eqnarray}
has appeared, where $\zeta$ is the Riemann zeta function.
Indeed, this constant has appeared in a variety of works spanning several decades and treating problems in a wide arrange of areas.
In short, \cite{chernoff1965sequential} and~\cite{siegmund1985sequential} study sequential statistical tests; \cite{siegmund1979corrected}, \cite{siegmund1982brownian}, \cite{coffman1998maximum} and~\cite{comtet2005precise} study approximating with Brownian motion certain events of random walks; \cite{asmussen1995discretization} and~\cite{calvin1995average} study discretized Brownian motion; \cite{chang1997ladder}, \cite{janssen2007cumulants} and~\cite{janssen2007lerch} study Gaussian walks; and \cite{broadie1997continuity}, \cite{broadie1999connecting} and~\cite{dia2011continuity} study option pricing in financial engineering.
As pointed out by~\cite{janssen2007lerch}, the connection between these works having the constant $\beta$ as common denominator has been unclear.

We claim that our collection of results, Theorems~\ref{theo1}, \ref{theo2} and~\ref{theo3} and Corollaries~\ref{corol1}, \ref{corol2} and~\ref{corol3}, connect these works in the literature.
The key fact buttressing our argument is the following result.
It shows that two of the limiting random variables we obtain in our results have the same mean equal to $\beta$, despite of coming from different but related problems.
\begin{proposition}\label{prop:beta}
Let $R$ be a standard two-sided Bessel(3) process, $W$ a standard Brownian motion, and $U$ a random variable uniformly distributed on the interval $(0, \, 1)$ which is independent of $R$ and $W$.
It holds that
\begin{eqnarray}\label{beta eq 1}
& \EE \left[ W \left( U+\min\{k \geq 0 : W(U+k)>0\} \right) \right] = \EE \left[ \min_{k \in \Z} R (U+k) \right] = -\frac{\zeta(1/2)}{\sqrt{2 \pi}}.
\end{eqnarray}
\end{proposition}
The second equality is proven in~\cite{asmussen1995discretization}; and the first one by \cite{broadie1997continuity} who compute the mean of the limiting overshoot of a Gaussian walk, whose limit distribution is given by Corollary~\ref{corol1}.

We now argue that each of the papers in which the constant $\beta$ has appeared can be viewed as an extension, particular case or application of one of results our results.
To support this argument, we now give an overview of this literature and organise it based on which of our results best explains why the constant appears in each paper.
We focus on briefly describing the problem the paper treats and how the constant $\beta$ appears.
Our intended objective is for this summary to draw the connection underlying these works in the literature.

\subsection{Papers related to our Theorem~\ref{theo1}}

Recall that Theorem~\ref{theo1} studies the discretizaton error of Brownian motion when simulating barrier (function) hitting events.

	\paragraph{\cite{chernoff1965sequential}.}
	This paper studies the problem of deciding the sign of a stream of iid normal random variables.
	It argues that this discrete problem reduces to one of monitoring if a Brownian motion sequentially observed at times at times $\{0, \delta, 2\delta, \ldots\}$, for $\delta>0$, ever crosses a given deterministic barrier function, say~$\tilde x_\delta (\cdot)$.
	Chernoff points out that additionally the original problem of deciding the sign of iid normal observations can be approximated by a continuous one of sequentially testing for the sign of a Brownian motion; and which moreover reduces to testing if a Brownian motion ever reaches another given deterministic barrier function, say $\tilde x (\cdot)$.
	The main result of \cite{chernoff1965sequential} is that both barriers $\tilde x (\cdot)$ and $\tilde x_\delta (\cdot)$ are related in the following way
	$$\tilde x (\cdot) = \tilde x_\delta (\cdot) + \beta \sqrt \delta + o(\sqrt \delta), \qquad \text{as }\delta \to 0,$$
	where the constant $\beta$ is as defined in~\eqref{def:beta}.
	
	This result is related to our Theorem~\ref{theo1} because in it we compare the barrier-hitting time of a Brownian motion to the barrier-hitting time for its discretized version on a time mesh of the type $\{0, \delta, 2\delta, \ldots\}$ with $\delta \to 0$.
	Moreover, note that by Proposition~\ref{prop:beta} the constant $\beta$ is the mean of one of the limitting random variables in Theorem~\ref{theo1}.
	With this we have that both results, i.e.~our theorem and Chernoff's result, establish that the distance between two barrier-hitting positions goes to zero at rate $\beta\sqrt\delta+o(\sqrt \delta)$ as $\delta \to 0$.

	\paragraph{\cite{broadie1997continuity}.}
	Broadie~et~al.~study the pricing of a barrier option when the stock price can only be monitored on a pre-specified regular time mesh.
	The pricing of the barrier involves the hitting time of a Brownian motion to a fixed barrier; however the precise hitting time is unknown as the Brownian motion is ``observed'' only on a regular mesh.
	This result is related to our Theorem~\ref{theo1} since in it we study the error of simulating barrier hits of Brownian motion when one can only observe the Brownian path on a regular mesh.
	We also remark that in~\cite{broadie1997continuity} the constant $\beta$ is derived from the mean overshoot of a Gaussian walk above a barrier that goes to infinity, which is also related to our Corollary~\ref{corol1}, as will be seen in Section~\ref{subsec:corol1}.

	\paragraph{\cite{dia2011continuity}.}
	Dia~et~al.~extend the work of \cite{broadie1997continuity} to price barrier options which additionally have positive jumps.
	We remark though that their main result is not proven by extending \cite{broadie1997continuity} but by using the results of \cite{asmussen1995discretization}, which are related to Theorem~\ref{theo2} as will be seen in Section~\ref{subsec:theo2}.

\subsection{Papers related to our Corollary~\ref{corol1}}\label{subsec:corol1}

\paragraph{\cite{siegmund1979corrected}.}
This paper studies approximating certain rare events of light-tailed random walks by an analogue event for Brownian motion.
In particular, \cite[Lemma~3]{siegmund1979corrected} studies the overshoot of a random walk above a barrier when the barrier goes to infinity.
This related to our Corollary~\ref{corol1}, since the second component of the pair~\eqref{theo:cor 1 exp} corresponds to the overshoot of the Gaussian walk above a barrier value $m$; and by Proposition~\ref{prop:beta} the mean of the limitting overshoot in the Gaussian walk case is equal to $\beta$.

\paragraph{\cite{siegmund1982brownian}.}
This paper extends Lemma~3 of~\cite{siegmund1979corrected}, which is related to our Corollary~\ref{corol1}.

\paragraph{\cite{siegmund1985sequential}.}
In Chapter~VIII, ``Random Walk and Renewal Theory'', Siegmund studies the overshoot of random walks from a renewal theory perspective.
He shows that there is an asymptotic ``equilibium'' limit overshoot distribution, and derive results in the line of~\cite[Lemma~3]{siegmund1979corrected} and~\cite[Theorem~1]{siegmund1982brownian}.

Additionally, Chapter~X, ``Corrected Brownian Approximations'', covers diffusion approximations results in the line of~\cite{siegmund1979corrected} and~\cite{siegmund1982brownian}.
Specifically, they derive diffusion approximation results by using the asymptotic overshoot results of Chapter~VIII.

\paragraph{\cite{chang1997ladder}.}
This paper considers the mean of the ladder height of a Gaussian walk, and obtains a series expansion of this quantity in terms of the drift term.
The constant $\beta$ appears as the coefficient of a term in the series expansion.
This result is tangentially related to our Corollary~\ref{corol1} in that our result describes the (asymptotic) overshoot of Gaussian walks, and these values can in turn be studied by reducing the walk to a sum of ladder heights of the random walk.
Moreover, by Proposition~\ref{prop:beta} we know that $\beta$ is the mean of the limit distribution of the overshoot in the Gaussian case.

\subsection{Papers related to our Theorem~\ref{theo2}}\label{subsec:theo2}

\paragraph{\cite{asmussen1995discretization}.}
This paper studies the convergence of the Euler discretization of one-dimensional \emph{reflected Brownian motion}.
Their main result is the weak convergence of the second component of the pair~\eqref{ch3:theo:min a} in our Theorem~\ref{theo2}.
The constant $\beta$ appears in this paper in \cite[Theorem 2]{asmussen1995discretization}, which in the context of our Theorem~\ref{theo2} establishes that
	$$\EE \left[ \sqrt n \left( \min_{k=0, \ldots,n}B(k/n) - \min_{0 \leq t \leq 1}B(t) \right) \right] \rightarrow \EE \left[ \min_{k \in \Z} R(k+U) \right] = -\frac{\zeta(1/2)}{\sqrt{2 \pi}}$$
	as $n \to \infty$.

\paragraph{\cite{calvin1995average}.}
This paper considers the minimization of a real function over a compact real interval and studies the performance of using discretization schemes to carry out the minimization.
To do this, Calvin argues that a reasonable benchmark is to assume that the objective function is a Brownian path.
The constant $\beta$ appears essentially because they show that
$$\EE \left[ \sqrt n \left( \min_{t \in \{ 0, 1/n, \ldots, 1\}} B(t) - \min_{t \in [0,1]} B(t) \right) \right] \rightarrow -\frac{\zeta(1/2)}{\sqrt{2 \pi}}$$
as $n \to \infty$.

\paragraph{\cite{broadie1999connecting}.}
	In this paper Broadie et~al.~study the pricing of lookback options when the stock price can only be monitored on a pre-specified regular time mesh.
	The pricing of the barrier involves the event of a Brownian motion going over a fixed barrier on a given finite time horizon.
	This event is put in terms of the maximum of the Brownian motion being greater than the barrier; however the precise maximum of the Brownian motion is unknown, as the Brownian motion is ``observed'' only at a regular mesh.
	More specifically, Broadie~et~al.~are interested in expressing the event $\{B^n(t)>y, \ \tau_b^n \leq t \}$ in terms of $B(t)$ and $\tau_b$; here, $B$ is a Brownian motion, $\tau_b$ is its hitting time to the barrier $b$, and $B^n$ and $\tau_b^n$ are the analogous discretized versions.
Their key result is that they show the approximation
\begin{eqnarray}\label{eq:broadie}
\PP \left( B^n(t)>y, \ \tau_b^n \leq t \right) = \PP \left( B(t)>y , \ \tau_{b + \sigma \beta/\sqrt n} \leq t \right) + o(1/\sqrt n),
\end{eqnarray}
for $t$ in the discretization mesh, where $\beta$ is as defined in~\eqref{def:beta}.
We remark that in view of our Theorem~\ref{theo2} and Proposition~\ref{prop:beta}, the asymptotic approximation of~\eqref{eq:broadie} can be \emph{heuristically} understood as, for $t$ in the discretization mesh,
\begin{eqnarray*}
\lefteqn{ \left\{ B^n(t)>y, \, \tau_b^n \leq t \right\} } \\
&& = \left\{B(t)>y, \, \max_{s \in [0,t]} B(s) \geq b + \frac{1}{\sqrt n} \sqrt n \left( \max_{s \in [0,t]} B(s)-\max_{s \in [0,t]} B^n(s) \right) \right\} \\
&&\approx \left\{B(t)>y, \, \max_{s \in [0,t]} B(s) \geq b + \frac{1}{\sqrt n} \EE \left[ \sqrt n \left( \max_{s \in [0,t]} B(s)-\max_{s \in [0,t]} B^n(s) \right) \right]\right\} \\
&&\approx \left\{B(t)>y, \, \max_{s \in [0,t]} B(s) \geq b + \frac{1}{\sqrt n} \sigma \beta \right\} = \left\{B(t)>y, \, \tau_{b + \sigma \beta/\sqrt n} \leq t \right\}.
\end{eqnarray*}
	
\subsection{Papers related to our Corollary~\ref{corol2}}

\paragraph{\cite{coffman1998maximum} and \cite{comtet2005precise}.}
	These papers consider random walks with symmetric step sizes and study their expected maximum up to time $n$.
	They do this by comparing it to the expected maximum of a Brownian motion up to time $n$, and then derive an asymptotic expression when $n \to \infty$.
	The constant $\beta$ appears when specializing to the particular case of Gaussian walks.
	Therefore these result are related to our Corollary~\ref{corol2} since in it we compare finite-horizon extremes of Gaussian walks to the analogue of Brownian motion, and by Proposition~\ref{prop:beta} we know that one of the limitting distributins has mean equal to $\beta$.

\subsection{Papers related to our Corollary~\ref{corol3}}

\paragraph{\cite{siegmund1979corrected}.}
This paper considers light-tailed random walks with negative drift and expresses their mean infinite-horizon maximum as an asymptotic expansion in terms of the drift of the random walk.
	The leading term is the mean of a global maximum of a Brownian motion with same drift as the random walk, minus the constant $\beta$, plus other terms that go to zero as the drift decreases.
	This result is related to our Corollary~\ref{corol3} since in it we compare the global extremes of Gaussian walks to the analogue for Brownian motion; and from Proposition~\ref{prop:beta}, the constant $\beta$ corresponds to the mean of one of the limitting random variables.

\paragraph{\cite{janssen2007cumulants,janssen2007lerch}.}
These papers study essentially the moments of the global minimum of a Gaussian walk with positive drift.
They compare these moments to the corresponding ones of the global minimum of a Brownian motion with the same positive drift.
They give a series expansion, in terms of the drift, of the difference of these two quantities and obtain that one of the leading terms is the constant $\beta$.
These results are related to our Corollary~\ref{corol3}, since there we compare the global extremes of Gaussian walks and Brownian motions, when both have strictly positive drift that decreases to zero.
Additionally, by Proposition~\ref{prop:beta} we know that $\beta$ is the first moment of one of the limitting random variables in Corollary~\ref{corol3}.


\section{Proof of Theorems~\ref{theo1}, \ref{theo2} and \ref{theo3}}\label{ch3:sec:proof}

In this section we show the proof of Theorems~\ref{theo1}, \ref{theo2} and~\ref{theo3}.
The main idea is to write the discretization errors as mappings of the original Brownian motion and then apply the continuous mapping theorem, or an argument of that type, to show weak convergence of the errors.
We do this by first, in Section~\ref{ch3:subsec:errors as maps}, writing the discretization errors as mappings of the original Brownian motion ``zoomed-in'' about the random time of interest.
Then, in Section~\ref{ch3:subsec:zoomed-in converge K} we show that the zoomed-in processes converge in distribution.
Lastly, in Section~\ref{ch3:subsec:proof} we tie all things together and prove Theorems~\ref{theo1}, \ref{theo2} and~\ref{theo3}.

For the sake of clarity of exposition, we briefly recall the notation used.
For a Brownian motion path $B$ with drift $\mu$ and variance $\sigma^2$ we consider its Euler discretization $B^n$ on the mesh $\{0, \ 1/n, \ 2/n, \ \ldots \}$ as $B^n(t) := B(\lfloor nt \rfloor / n)$ for all $t \geq 0$.
Also, recall that the times $\tau_b$, $T_{\min,a}$ and $T_{\min,\infty}$ are defined as follows:
\begin{eqnarray*}
\tau_b &:=& \inf \left\{ t \in (0, \infty) : B(t) \geq b(t) \right\}, \\
T_{\min,a} &:=& \inf \left\{ t \in [0,a] : B(t) = \inf_{u \in [0,a]} B(u) \right\}, \\
T_{\text{min},\infty} &:=& \inf \left\{ t \in [0,\infty) : B(t) = \inf_{u \in [0,\infty)} B(u) \right\};
\end{eqnarray*}
and also that their discretized counterparts $\tau_b^n$, $T_{\min,a}^n$ and $T_{\text{min},\infty}^n$ are defined analogously by replacing $B$ by $B^n$ in the previous definitions.
Recall too that by Assumption (H$_b$) the function $b=(b(t) : t \geq 0)$ is continuous and nondecreasing on $\R_+$, continuously differentiable on $\R_+ \!\! \setminus \!\! \{ 0 \}$, and with $b(0) \geq 0$.
We onwards assume that $B$ is actually a two-sided Brownian motion $B = \left( B(t) : t \in \R \right)$ with $B(0)=0$.

\subsection{Discretization errors as mappings of {zoomed-in} processes}\label{ch3:subsec:errors as maps}

In this section we show that the discretization error expressions in Theorems~\ref{theo1}, \ref{theo2} and \ref{theo3} can be rewritten as mappings of certain centerings and scalings of the original Brownian motion~$B$; we call these the {\it {zoomed-in}} processes.
We remark that these processes are separate entities from the Euler discretization $B^n$ of the Brownian motion $B$.

\begin{definition}\label{ch3:def:zoomed-in}
\begin{enumerate}
	\item Under $\{ \tau_b < \infty \}$, for $\m > 0$ define
	\begin{eqnarray}
	\UU{\m}{\hit,b} &:=& \lceil \m \tau_b \rceil - \m \tau_b,
	\end{eqnarray}
	the {\it {zoomed-in}} process $\ZZ{\m}{\hit,b} = \left( \ZZ{\m}{\hit,b} (s) : s \in \R \right)$ as
	\begin{eqnarray}\label{ch3:def:scaled hit}
	\ZZ{\m}{\hit,b} (s) &:=& \sqrt \m \left( B \left( \tau_b + \frac{s}{\m} \right) - B \left( \tau_b \right) \right) , \qquad s \in \R,
	\end{eqnarray}
	and
	\begin{eqnarray}\label{ch3:def:scaled b}
	\bb{\m}{\hit,B} (s) &:=& \sqrt \m \left( b \left( \tau_b + \frac{s}{\m} \right) - b \left( \tau_b \right) \right), \qquad s \in \R_+.
	\end{eqnarray}
	
	\item For $\m > 0$ define
	\begin{eqnarray}
	\UU{\m}{\min,a} &:=& \lceil \m T_{\min,a} \rceil - \m T_{\min,a}
	\end{eqnarray}
	and the {\it {zoomed-in}} process $\ZZ{\m}{\min,a} = \left( \ZZ{\m}{\min,a} (s) : s \in \R \right)$ as
	\begin{eqnarray}\label{ch3:def:scaled min a}
	\ZZ{\m}{\min,a} (s) &:=& \sqrt \m \left( B \left( T_{\min,a} + \frac{s}{\m} \right) - B \left( T_{\min,a} \right) \right) , \qquad s \in \R.
	\end{eqnarray}
	
	\item Under $\{ T_{\min,\infty} < \infty \}$, for $\m > 0$ define
	\begin{eqnarray}
	\UU{\m}{\min,\infty} &:=& \lceil \m T_{\min,\infty} \rceil - \m T_{\min,\infty}
	\end{eqnarray}
	and the {\it {zoomed-in}} process $\ZZ{\m}{\min,\infty} = \left( \ZZ{\m}{\min,\infty} (s) : s \in \R \right)$ as
	\begin{eqnarray}\label{ch3:def:scaled min infty}
	\ZZ{\m}{\min,\infty} (s) &:=& \sqrt \m \left( B \left( T_{\min,\infty} + \frac{s}{\m} \right) - B \left( T_{\min,\infty} \right) \right) , \qquad s \in \R.
	\end{eqnarray}
\end{enumerate}
\end{definition}

Intuitively, as $\m$ grows, the processes \eqref{ch3:def:scaled hit}, \eqref{ch3:def:scaled min a} and~\eqref{ch3:def:scaled min infty} can be understood as ``zooming-in's'' of the Brownian path respectively about the points $(\tau_b, \ B(\tau_b))$, $(T_{\min,a}, \ B(T_{\min,a}))$ and $(T_{\min,\infty}, \ B(T_{\min,\infty}))$.

We now rewrite the discretization errors in Theorems~\ref{theo1}, \ref{theo2} and \ref{theo3} as mappings of the processes $\ZZ{\m}{\hit,b}$, $\ZZ{\m}{\min,a}$ and $\ZZ{\m}{\min,\infty}$, the times $T_{\min,a}$ and $T_{\min,\infty}$, and the random variables $\UU{\m}{\hit,b}$, $\UU{\m}{\min,a}$ and $\UU{\m}{\min,\infty}$.

\begin{lemma}\label{ch3:lemma:decomp}
\begin{enumerate}
	\item Under $\{ \tau_b < \infty , \ \tau_b^n < \infty \}$, almost surely we have
	\begin{eqnarray}
	\left( \begin{array}{c} n \left( \tau_b^n - \tau_b \right) \\ \sqrt n \left( B^n(\tau_b^n) - B(\tau_b) \right) \end{array} \right) &=&  E_{\hit} \left( \UU{n}{\hit,b} , \ \ZZ{n}{\hit,b}, \ \bb{n}{\hit,B} \right),
	\end{eqnarray}
	where the mapping $E_{\hit} : \R \times \CC(\R) \times \CC(\R_+) \to (\R \cup \{ +\infty \}) \times (\R \cup \{ \partial \})$ is defined as
	\begin{eqnarray}\label{ch3:def:E hit}
	E_{\hit} (u, f, g) &:=&  \left( \begin{array}{r} u+ \min\left\{k  \in \Z_+ : f(u+k) > g(u+k) \right\} \ \, \\ f\left( u + \min\left\{k  \in \Z_+ : f(u+k) > g(u+k) \right\} \right) \end{array}  \right),
	\end{eqnarray}
	where for completeness we use the convention $\min \emptyset := +\infty$ and for $f \in \CC(\R)$ we define $f(+\infty) := \partial$.
	
	\item It holds that
	\begin{eqnarray}
	\left(  \begin{array}{c} n \left( T_{\min,a}^n - T_{\min,a} \right) \\ \sqrt n \left( B^n(T_{\min,a}^n) - B(T_{\min,a}) \right) \end{array}  \right) &=& E_{\min,a}^{(n)} \left( \UU{n}{\min,a} ,\ \ZZ{n}{\min,a} , \ T_{\min,a} \right)
	\end{eqnarray}
	where for all $\m \in \Z_+$ the mapping $E_{\min,a}^{(\m)} : \R \times \CC (\R) \times (0,a) \to \R^2$ is defined as
	\begin{eqnarray}\label{ch3:def:E min a}
	E_{\min,a}^{(\m)} (u, f, t) &:=& \left( \begin{array}{r} u + \arginf_{k \in \Z \cap [-\lceil \m t \rceil, \m a-\lceil \m t \rceil]} f(u+k) \\ \inf_{k \in \Z \cap [-\lceil \m t \rceil, \m a-\lceil \m t \rceil]} f(u+k) \end{array} \right),
	\end{eqnarray}
	where we have abused notation and actually denote $\arginf_{s \in N} g(s) := \inf\{ s \in N : g(s) = \inf_{u \in N} g(u) \}$ for $g \in \CC(\R)$, $N \subseteq \R$ compact and $\inf\emptyset := +\infty$.
	
	\item Under $\{ T_{\min,\infty} < \infty \}$ we have that
	\begin{eqnarray}
	\left( \begin{array}{c} n \left( T_{\min,\infty}^n - T_{\min,\infty} \right) \\ \sqrt n \left( B^n(T_{\min,\infty}^n) - B(T_{\min,\infty}) \right) \end{array} \right) &=& E_{\min,\infty}^{(n)} \left( \UU{n}{\min,\infty}, \ \ZZ{n}{\min,\infty} , \ T_{\min,\infty} \right)
	\end{eqnarray}
	where for all $\m \in \Z_+$ the mapping $E_{\min,\infty}^{(\m)} : \R \times \CC_{\inf} (\R) \times \R \to (\R \cup \{+\infty\}) \times \R$ is defined as
	\begin{eqnarray}\label{ch3:def:E min infty}
	E_{\min,\infty}^{(\m)} (u, f, t) &:=& \left( \begin{array}{r} u+ \arginf_{k \in \Z \cap [-\lceil \m t \rceil, \infty)} f(u+k) \\ \inf_{k \in \Z \cap [-\lceil \m t \rceil, \infty)} f(u+k) \end{array} \right),
	\end{eqnarray}
	where, for completeness, we denote $\CC_{\inf+} (\R) :=\{f \in \CC(\R) : \lim_{t \to \infty} f(t) = \infty\}$ and $\arginf_{s \in N} g(s) := \inf\{t \in N : g(t) = \inf_{s \in N} g(s) \}$ for $g \in \CC_{\inf+} (\R)$ and $N \subset \R$ closed, with the convention $\inf \emptyset = \infty$.
\end{enumerate}
\end{lemma}

The proof of the previous result is a straightforward but non-illuminating calculation, so it is deferred to Appendix~\ref{ch3:sec:appendix:1}.

\subsection{Weak convergence of the {zoomed-in} processes}\label{ch3:subsec:zoomed-in converge K}

In this section we show that the zoomed-in processes defined in Definition~\ref{ch3:def:zoomed-in} converge in distribution.
We endow the space $\CC (\R)$ of continuous real functions on $\R$ with the metric {\it of uniform convergence over compact sets}, say $\nK \cdot$, defined for all $f$ in $\CC (\R)$ as
$$\nK{f} := \sum_{\A=1}^\infty \frac{1}{2^\A} \min\left\{ 1 \ , \ \sup_{t \in [-\A, \A]} |f(t)| \right\}.$$
It holds that $\nK{f} = 0$ if and only if $\sup_{s \in [-\A, \A]} |f(s)| = 0$ for all $\A > 0$ finite, which motivates the name of the metric.
The space $\CC(\R)$ endowed with the topology generated by this metric is a {\it Polish} space, i.e., a complete and separable topological space; see e.g.~\cite{whitt1970weak} for further details.
We also use a one-sided version of the metric space $(\CC (\R), \nK{\cdot})$, which is defined analogously to the two-sided version $(\CC (\R), \nK{\cdot})$.

\begin{lemma}\label{ch3:lemma:scalings ucocs}
\begin{enumerate}
	\item Given $t>0$ the process $\ZZ{\m}{\hit,b}$ conditioned on  $\{ \tau_b = t \}$ converges in distribution on $(\CC (\R) , \ \nK{\cdot})$ to $\left( \sigma R(-s) : s \leq 0 \ ; \ \sigma W(s) : s \geq 0 \right)$ as $\m \to \infty$, where $W$ is a standard Brownian motion and $R$ is a standard Bessel(3) process independent of $W$.
	
	\item Given $t \in (0,a)$, $\lo < 0$ and $y > 0$, conditioned on $\{T_{\min,a} = t, \ B(T_{\min,a}) = \lo , \ B(a) = \lo+y\}$ the process $\ZZ{\m}{\min,a}$ converges in distribution on $(\CC (\R) , \ \nK{\cdot})$ to $\sigma R$ as $\m \to \infty$, where $R = \left( R (s) : s \in \R \right)$ is a two-sided standard Bessel(3) process.
	
	\item Given $t>0$ and $\lo<0$, conditioned on the event $\{T_{\min,\infty} = t, \ B(T_{\min,\infty}) = \lo\}$ the process $\ZZ{\m}{\min,\infty}$ converges in distribution on $(\CC (\R) , \ \nK{\cdot})$ to $\sigma R$ as $\m \to \infty$, where $R = \left( R (s) : s \in \R \right)$ is a standard two-sided Bessel(3) process.
\end{enumerate}
\end{lemma}

We remark that in all three parts of Lemma~\ref{ch3:lemma:scalings ucocs} the convergence in distribution also holds unconditionally, since the limiting processes do not depend on the values of the conditioning.
Note that none of these results explicitly depend on the drift $\mu$ of the Brownian motion $B$.

\begin{proof}[Proof of Lemma~\ref{ch3:lemma:scalings ucocs}]
For the sake of clarity of the exposition, we always assume that the treatment is conditioned on the events $\{ \tau_b < \infty , \ \tau_b^n < \infty \}$ and $\{ T_{\min,\infty} < \infty \}$ when dealing with the processes $\ZZ{\m}{\hit,b}$ and $\ZZ{\m}{\min,\infty}$ respectively.

We start by proving (i).
By the strong Markov property the process $( \ZZ{1}{\hit,b} (s) = B(t+s)-b : s \geq 0 )$ conditioned on $\{ \tau_b = t \}$ is distributed as a Brownian motion with drift $\mu$ and variance $\sigma^2$.
In particular, $( \ZZ{1}{\hit,b} (s) : s \geq 0 )$ is independent of $( \ZZ{1}{\hit,b} (s) : s \leq 0 )$.
Then, by Brownian scaling $( \ZZ{\m}{\hit,b} (s) : s \geq 0)$ is equal in distribution to a Brownian motion with drift $\mu/\sqrt \m$ and variance $\sigma^2$.
Since it converges almost surely to $\sigma W$ on $(\CC (\R_+), \ \nK{\cdot})$ as $\m \to \infty$, where $W$ is a standard Brownian motion, then it also converges weakly on $(\CC (\R_+), \ \nK{\cdot})$.

On the other hand, the process $(-\ZZ{1}{\hit,b} (-s) = b - B(\tau_b-s) : s \in [0,t])$ conditioned on $\{ \tau_b = t \}$ is distributed as a Bessel(3) process conditioned on being at $b$ at time $t$, see~\cite[Theorem~3.4]{williams1974path}.
Applying~\cite[Lemma~1]{asmussen1995discretization} we obtain that $( -\sqrt \m \ZZ{1}{\hit,b} (-s / \m) : s \geq 0 )$ converges weakly to $\sigma R$ on $(\CC (\R_+), \ \nK{\cdot})$ as $\m \to \infty$, where $R$ is a Bessel(3) process independent of $W$.
This proves part (i).

Part (ii) corresponds to~\cite[Lemma~1]{asmussen1995discretization}.

Finally, we prove (iii).
The process $( \ZZ{1}{\min,\infty}(s) = B(T_{\min,\infty}+s)-B(T_{\min,\infty}) : s \geq 0 )$ conditioned on the event $\{T_{\min,\infty} = t, \ B(T_{\min,\infty}) = \lo \}$ is distributed as a Bessel(3) process with drift $\mu$, see~\cite[Corollary~3]{rogers1981markov}.
Then, by Brownian scaling of the Bessel processes, $( \sqrt \m \ZZ{1}{\min,\infty}(s / \m) = \ZZ{\m}{\min,\infty}(s) : s \geq 0 )$ is distributed as a Bessel(3) process with drift $\mu/\sqrt \m$.
Such a process converges almost surely in $(\CC (\R_+) , \ \nK{\cdot})$ to a Bessel(3) process with no drift, say to $R_+$, so it also converges in distribution to $R_+$.

On the other hand, by~\cite[Theorems~2.1 and~3.4]{williams1974path} we have that the process $( \ZZ{1}{\min,\infty}(-s) : s \in [0,t] )$ conditioned on $\{T_{\min,\infty} = t, \ B(T_{\min,\infty}) = \lo \}$ is distributed as a Bessel(3) process conditioned on being at $-\lo$ at time $t$.
We conclude using~\cite[Lemma~1]{asmussen1995discretization} that $( \sqrt \m \ZZ{1}{\min,\infty} (-s / \m) = \ZZ{\m}{\min,\infty} (-s) : s \geq 0 )$ converges weakly to $\sigma R_-$ on $(\CC (\R_+), \ \nK{\cdot})$ as $\m \to \infty$, where $R_-$ is a Bessel(3) process independent of $R_+$.
This proves (iii).
\end{proof}

\subsection{Proofs}\label{ch3:subsec:proof} 

In this section we prove Theorems~\ref{theo1}, \ref{theo2} and \ref{theo3}.
The idea is to apply the continuous mapping theorem, see~\cite[Chapter~3.4]{whitt2002stochastic}.
We do this inspired by Lemma~\ref{ch3:lemma:decomp}, which shows that the errors are mappings of the zoomed-in processes $\ZZ{n}{\hit,b}$, $\ZZ{n}{\min,a}$ and $\ZZ{n}{\min,\infty}$ and other random variables, and Lemma~\ref{ch3:lemma:scalings ucocs}, which shows weak convergence of the zoomed-in processes.

The following result shows the weak convergence of the random variables $\UU{n}{\hit,b}$, $\UU{n}{\hit,b}$ and~$\UU{n}{\hit,b}$ to uniform random variables.
This result motivates the weak convergence of the pairs $( \UU{n}{\hit,b} , \ \ZZ{n}{\hit,b} )$, $( \UU{n}{\min,a} , \ \ZZ{n}{\min,a} )$ and $( \UU{n}{\min,\infty} , \ \ZZ{n}{\min,\infty} )$ as $n \to \infty$, with all the limiting distributions independent of the random variables $\tau_b$, $T_{\min,a}$ and $T_{\min,\infty}$.
We defer its proof to Appendix~\ref{ch3:sec:appendix:3}.

\begin{lemma}\label{ch3:lemma:U convergence}
Consider a nonnegative random variable $T$ which has a distribution that is absolutely continuous with respect to the Lebesgue measure.
Then as $n \to \infty$, $\lceil n T \rceil -nT$ converges in distribution to a uniformly distributed random variable on $(0,1)$, which moreover is independent of $T$.
\end{lemma}

With the previous result we are now able to prove the main results of this paper.

\begin{proof}[Proof of Theorem~\ref{theo1}]
First recall that by Lemma~\ref{ch3:lemma:decomp}, conditioned on $\{ \tau_b < \infty , \ \tau_b^n < \infty \}$ we have that
$$(n \left( \tau_b^n - \tau_b \right) , \ \sqrt n \left( B^n(\tau_b^n) - B(\tau_b) \right) ) = E_{\hit} \left( \UU{n}{\hit,b} , \ \ZZ{n}{\hit,b} , \ \bb{n}{\hit,B} \right).$$
The plan of the proof is to first show that the triplet
\begin{eqnarray}\label{ch3:eq:triplet Ehit}
& \left( \UU{n}{\hit,b} , \ \ZZ{n}{\hit,b} , \ \bb{n}{\hit,B} \right)
\end{eqnarray}
converges in distribution, then show that the function $E_{\hit}$ is continuous, and conclude the desired convergence of the errors using the continuous mapping theorem.
We will use the metric of uniform convergence on compact sets $\nK \cdot$ for the weak convergence of $\ZZ{n}{\hit,b}$ and for the continuity of the mapping $E_{\hit}$.

We first argue that the triplet in \eqref{ch3:eq:triplet Ehit} converges in distribution on $(\R \times \CC (\R) \times \CC (\R_+) , |\cdot| \times \nK{\cdot} \times \nK{\cdot})$ where, recall, $\nK \cdot$ is the metric of uniform convergence on compact sets.
Indeed, by Lemmas~\ref{ch3:lemma:scalings ucocs} and~\ref{ch3:lemma:U convergence} conditional on $\{ \tau_b < \infty \}$ the pair $(\UU{n}{\hit,b} , \ \ZZ{n}{\hit,b} )$ converges in distribution to $(U, \ (-\sigma R(-s): s \leq 0 ; \sigma W(s) : s \geq 0))$, where the weak convergence of $\ZZ{n}{\hit,b}$ is on $(\CC (\R) , \nK{\cdot})$.
Here, $R$ and $W$ are standard Bessel(3) and Brownian motion processes, respectively; $U$ is uniformly distributed on $(0,1)$; and  $R$, $W$, $U$ and $\tau_b$ are all independent.
Additionally, the function $\bb{n}{\hit,B}$ almost surely converges to $0$ on $(\CC (\R_+) , \nK{\cdot})$, i.e., to the zero function.
Indeed, $\bb{n}{\hit,B}$ is continuously differentiable on $\R_+ \!\! \setminus \!\! \{ 0 \}$ by Assumption (H$_b$), and $\tau_b > 0$ almost surely, so for all $k>0$ it holds almost surely that
$$\sup_{t \in [0, k]} \left| \bb{n}{\hit,B} \right| = \sqrt n \sup_{t \in [0, k]} \left| b\left(\tau_b + \frac{t}{n} \right) - b(\tau_b) \right| = \sqrt n \sup_{t \in [0, k]} \left| b' \left(\tau_b + \frac{\xi}{n} \right) \frac{t}{n} \right| \to 0$$
as $n \to \infty$, where the last equality holds for some $\xi \in (0, k)$ by the mean value theorem.
With this, we conclude that the triplet~\eqref{ch3:eq:triplet Ehit} converges in distribution as $n \to \infty$ to the triplet
\begin{eqnarray}\label{ch3:limit triplet Ehit}
& (U, \ (-R(-s): s \leq 0 ; W(s) : s \geq 0), \ 0)
\end{eqnarray}
on $(\R \times \CC (\R) \times \CC (\R_+) , |\cdot| \times \nK{\cdot} \times \nK{\cdot})$.

The next step is to show that the mapping $E_{\hit}$ defined in \eqref{ch3:def:E hit} is continuous on a measurable set, say $S_{\hit}$, containing the support of the limiting random variable \eqref{ch3:limit triplet Ehit}.
We consider the set $S_{\hit}$ consisting on the triplets $(u,w,0) \in (0,1) \times \CC(\R_+) \times \{0\}$ satisfying $\sup_{k \in \Z_+} w(u+k) > 0$ and $w(u+k) \neq 0$ for all $k \in \Z_+$.
Clearly it is measurable and the support of \eqref{ch3:limit triplet Ehit} is contained in $S_{\hit}$.
In particular, the mapping $E_{\hit}$ on $S_{\hit}$ takes values in $\R^2$.
Note now that by composition of continuous functions it is sufficient to only show the continuity of the mapping
\begin{eqnarray}\label{ch3:function proof 1}
(u, w, 0) \in S_{\hit} &\mapsto& \min\{k  \in \Z_+ : w(u+k) > 0\}.
\end{eqnarray}
To see that the function \eqref{ch3:function proof 1} is continuous, first take any $(u, w, 0) \in S_{\hit}$ and a sequence $((\UU{n}{\min,a}, w_n) : n \geq 0)$ in $\R \times \CC(\R_+)$ such that $\nK{ w_n - w } + |\UU{n}{\min,a} - u| \to 0$ as $n \to \infty$.
Denote $K^* := \min\{k  \in \Z_+ : w(u+k) > 0\}$ and $\delta^* := \min \{ |w(u+k)| :  k = 0, \ldots, K^* \}$, and note that $\delta^* > 0$ by definition of $S_{\hit}$.
Since $(\UU{n}{\min,a}, w_n) \to (u, w)$ on $| \cdot | \times \nK{\cdot}$ then in particular $\sup_{k \in \Z_+ \cap [0, K^*]} \left| w(u+k) - w_n(\UU{n}{\min,a}+k) \right| < \delta^*$ for all $n \geq n^*$, for some $n^*$ sufficiently large.
In particular, for all $n \geq n^*$ the values $w_n(\UU{n}{\min,a}+k)$, $k = 0, \ldots, n^*$, are all different from zero and have the same sign of $w(u+k)$, $k = 0, \ldots, n^*$, respectively.
This implies that $\min\{k  \in \Z_+ : w_n(\UU{n}{\min,a}+k) > 0\} = K^* = \min\{k  \in \Z_+ : w(u+k) > 0\}$ for all $n \geq n^*$.
We have shown that the mapping \eqref{ch3:function proof 1}, and thus $E_{\hit}$, is continuous on $S_{\hit}$, which contains the support of the limiting random variable \eqref{ch3:limit triplet Ehit}.

It follows that applying the continuous mapping theorem, see~\cite[Theorem~3.4.3]{whitt2002stochastic},  the random variable $E_{\hit} ( \UU{n}{\hit,b} , \ \ZZ{n}{\hit,b} , \ \bb{\m}{\hit,B} )$ converges in distribution to $E_{\hit} ( U , \ (-\sigma R , \sigma W) , \ 0)$ as $n \to \infty$.
Lastly, note that $\PP (\tau_b^n < \infty | \tau_b < \infty ) \to 1$ as $n \to \infty$, so we can drop the condition $\tau_b^n < \infty $.
This concludes the proof of the Theorem~\ref{theo1}.
\end{proof}

\begin{proof}[Proof of Theorems~\ref{theo2} and~\ref{theo3}]
We show only the proof of Theorem~\ref{theo2}, since the proof of Theorem~\ref{theo3} is analogous.

To prove the joint convergence of the normalized discretization errors $ n ( T_{\min,a}^n - T_{\min,a} )$ and $\sqrt n ( B^n(T_{\min,a}^n) - B(T_{\min,a}) )$ recall first that by Lemma~\ref{ch3:lemma:decomp} they can be written as
\begin{eqnarray*}
\left(  \begin{array}{c} n \left( T_{\min,a}^n - T_{\min,a} \right) \\ \sqrt n \left( B^n(T_{\min,a}^n) - B(T_{\min,a}) \right) \end{array}  \right) = E_{\min,a}^{(n)} \left( \UU{n}{\min,a} ,\ \ZZ{n}{\min,a} , \ T_{\min,a} \right).
\end{eqnarray*}
Inspired on this we first show that the following vector converges in distribution as $n \to \infty$:
\begin{eqnarray}\label{ch3:eq:vector E min}
& \left( \UU{n}{\min,a} ,\ \ZZ{n}{\min,a} , \ T_{\min,a} \right).
\end{eqnarray}
We would then like to conclude the joint convergence of the normalized errors by using the generalized continuous mapping theorem, see~\cite[Theorem~3.4.4]{whitt2002stochastic}.
This considers showing that, in a sense, the mapping $\lim_{n \to \infty} E_{\min,a}^{(n)}$ is continuous for the metric $\nK \cdot$ on the support of the limiting distribution of \eqref{ch3:eq:vector E min}; however this is not true.
Nonetheless, we show that this problem can be circumvented by first restricting to compact time intervals, then using there the continuous mapping theorem, and then increasing the size of the compact time interval.

We start by arguing that as $n \to \infty$ the random variable in \eqref{ch3:eq:vector E min} converges in distribution to $(U, \ \sigma R, \ T_{\min,a})$, where $R$ is a standard Bessel(3) process, $U$ is uniformly distributed on $(0,1)$, and  $R$, $U$ and $T_{\min,a}$ are all independent.
Indeed, by Lemmas~\ref{ch3:lemma:scalings ucocs} and~\ref{ch3:lemma:U convergence}, as $n \to \infty$ the pair $( \UU{n}{\min,a} ,\ \ZZ{n}{\min,a} )$ converges in distribution to $(U, \ \sigma R)$, where the weak convergence of $\ZZ{n}{\min,b}$ is on $(\CC (\R) , \nK{\cdot})$, and moreover $U$ and $R$ are independent of $T_{\min,a}$.

The rest of the proof consists on showing that the following limit in distribution holds
\begin{eqnarray}\label{ch3:lim A n 3}
\lim_{n \to \infty} \ E_{\min,a}^{(n)} \left( \UU{n}{\min,a} ,\ \ZZ{n}{\min,a} , \ T_{\min,a} \right)  &=& E_{\min,a}^{(\infty)} \left( U ,\ \sigma R , \ T_{\min,a} \right),
\end{eqnarray}
where the mapping $E_{\min,a}^{(\infty)}$ is defined as
\begin{eqnarray}\label{ch3:def:E min inf}
E_{\min,a}^{(\infty)} (u, f, t) &:=& \left( u + \arginf_{k \in \Z} f(u+k) , \ \inf_{k \in \Z} f(u+k) \right),
\end{eqnarray}
where $f$ takes values in $\CC_{\inf\pm} (\R) :=\{f \in \CC(\R) : \lim_{t \to \pm \infty} f(t) = \infty\}$.
By Lemma~\ref{ch3:lemma:decomp} this would conclude the proof of Theorem~\ref{theo2}.
To prove the limit~\eqref{ch3:lim A n 3} we first restrict, in a sense, the mappings $E_{\min,a}^{(\m)}$ to compact time intervals of the form $[-\A,\A]$, then prove the weak convergence there, and then take the limit $\A \to \infty$.

Define for all $\m \in \Z_+$ and $\A>0$ the mapping $E_{\min,a}^{(\m,\A)} : (0,1) \times \CC (\R) \times (0,a) \to \R^2$ as
	\begin{eqnarray}\label{ch3:def:E min a A}
	E_{\min,a}^{(\m,\A)} (u, f, t) &:=& \left( \begin{array}{r} u + \arginf_{k \in \Z \cap [-\lceil \m t \rceil, \ \m a-\lceil \m t \rceil] \cap [-\A,\A]} f(u+k) \\ \inf_{k \in \Z \cap [-\lceil \m t \rceil, \ \m a-\lceil \m t \rceil] \cap [-\A,\A]} f(u+k) \end{array} \right),
	\end{eqnarray}
	where we have abused notation and actually denote $\arginf_{s \in N} g(s) := \inf\{ s \in N : g(s) = \inf_{u \in N} g(u) \}$ for $g \in \CC(\R)$, $N \subseteq \R$ compact and $\inf\emptyset := +\infty$.
By the generalized continuous mapping theorem, see~\cite[Theorem~3.4.4]{whitt2002stochastic}, as $n \to \infty$ the vector $E_{\min,a}^{(n,\A)} \left( \UU{n}{\min,a} ,\ \ZZ{n}{\min,a} , \ T_{\min,a} \right)$ converges in distribution  to $E_{\min,a}^{(\infty,\A)} \left( U ,\ \sigma R , \ T_{\min,a} \right)$, where
\begin{eqnarray}\label{ch3:def:E min infty a A}
E_{\min,a}^{(\infty,\A)} \left( u ,\ r , \ t \right) &:=& \left( u + \argmin_{k \in \Z \cap [-\A,\A]} r(u+k) , \ \min_{k \in \Z \cap [-\A,\A]} r(u+k) \right).
\end{eqnarray}
Indeed, for all $(u,r,t) \in (0,1) \times \CC (\R) \times (0,a)$ such that all the values $\{ r(u+k) : k \in \Z \}$ are different it holds that $E_{\min,a}^{(n,\A)} \left( u_n ,\ r_n , \ t_n \right) \to E_{\min,a}^{(\infty,\A)} \left( u ,\ r , \ t \right)$ for any sequence $(u_n,r_n,t_n)$ in $(0,1) \times \CC (\R) \times (0,a)$ such that $u_n \to u$, $r_n \to r$ and $t_n \to t$, where the convergence of $r_n$ is with the norm $\nK \cdot$.
This implies that, with probability one, the limiting random variable $(U, \ \sigma R, \ T_{\min,a})$ does not take values on the set where $E_{\min,a}^{(\infty,\A)}$ is discontinuous; therefore the generalized continuous mapping theorem applies and we obtain the desired convergence in distribution.

Now note that the following limit in distribution holds 
$$\lim_{\A \to \infty} \ E_{\min,a}^{(\infty,\A)} \left( U ,\ \sigma R , \ T_{\min,a} \right) = E_{\min,a}^{(\infty)} \left( U ,\ \sigma R , \ T_{\min,a} \right),$$
with $E_{\min,a}^{(\infty)}$ is as defined in~\eqref{ch3:def:E min inf}, because the convergence actually holds almost surely.
Therefore, abusing notation we conclude that
\begin{eqnarray}\label{ch3:lim A n}
\lim_{\A \to \infty} \, \lim_{n \to \infty} \, E_{\min,a}^{(n,\A)} \left( \UU{n}{\min,a} ,\, \ZZ{n}{\min,a} , \, T_{\min,a} \right) &=& E_{\min,a}^{(\infty)} \left( U ,\ \sigma R , \ T_{\min,a} \right),
\end{eqnarray}
where both limits and the equality are in distribution.

On the other hand, for all $n$
\begin{eqnarray}\label{ch3:lim A n 2}
\lim_{\A \to \infty} E_{\min,a}^{(n,A)} \left( \UU{n}{\min,a} ,\ \ZZ{n}{\min,a} , \ T_{\min,a} \right) &=& E_{\min,a}^{(n)} \left( \UU{n}{\min,a} ,\ \ZZ{n}{\min,a} , \ T_{\min,a} \right)
\end{eqnarray}
in distribution, since the convergence holds almost surely.

We now show that, in a sense, the following interchange of limits holds
$$\lim_{n \to \infty} \, \lim_{\A \to \infty} \, E_{\min,a}^{(n,A)} \left( \UU{n}{\min,a} ,\, \ZZ{n}{\min,a} , \, T_{\min,a} \right) = \lim_{\A \to \infty} \, \lim_{n \to \infty} \, E_{\min,a}^{(n,A)} \left( \UU{n}{\min,a} ,\, \ZZ{n}{\min,a} , \, T_{\min,a} \right),$$
where the limits and the equality are in distribution; note that this would conclude the limit in distribution~\eqref{ch3:lim A n 3}.
The latter interchange of limits, we will see, is a consequence of the limit
\begin{align}\label{ch3:lim:AGP95 lemma 2}
& \lim_{\A \to \infty} \limsup_{n \to \infty} \PP \left( \min_{k \in \Z \cap [-\lceil n T_{\min,a} \rceil, \ n a-\lceil n T_{\min,a} \rceil] \cap [-\A,\A]^\mathsf{C}} \ZZ{n}{\min,a} \left( k + \UU{n}{\min,a} \right) \! < \! \ttt \right) = 0,
\end{align}
which holds for all $\ttt \in \R$.
Indeed, \eqref{ch3:lim:AGP95 lemma 2} is just Lemma~4 of~\cite{asmussen1995discretization}, which is easily checked by using the definitions of $\ZZ{n}{\min,a}$ and $\UU{n}{\min,a}$ in Definition~\ref{ch3:def:zoomed-in}.
The limit~\eqref{ch3:lim:AGP95 lemma 2} in turn implies that the $\limsup_{\A \to \infty}$ of the $\limsup_{n \to \infty}$ of the following probability converges to zero:
\begin{align}\label{ch3:lim:AGP95 lemma 2 A}
& \PP \left( \UU{n}{\min,a} + \argmin_{k \in \Z \cap [-\lceil n T_{\min,a} \rceil, \ n a-\lceil n T_{\min,a} \rceil]} \ZZ{n}{\min,a} \left( k + \UU{n}{\min,a} \right) \in [-\A,\A]^\mathsf{C} \right).
\end{align}
Indeed, for all $\xi \in \R$ we have
\begin{eqnarray*}
\lefteqn{\PP \left( \UU{n}{\min,a} + \argmin_{k \in \Z \cap [-\lceil n T_{\min,a} \rceil, \ n a-\lceil n T_{\min,a} \rceil]} \ZZ{n}{\min,a}(k + \UU{n}{\min,a}) \in [-\A,\A]^\mathsf{C} \right) } \\
&& = \PP \left( \min_{k \in \Z \cap [-\lceil n T_{\min,a} \rceil, \ n a-\lceil n T_{\min,a} \rceil] \cap [-\A,\A]^\mathsf{C}} \ZZ{n}{\min,a}(k + \UU{n}{\min,a}) \right. \\
&& \qquad\qquad \left. {} < \min_{k \in \Z \cap [-\lceil n T_{\min,a} \rceil, \ n a-\lceil n T_{\min,a} \rceil] \cap [-\A,\A]} \ZZ{n}{\min,a}(k + \UU{n}{\min,a}) \right) \\
&& = \PP \left( \ldots, \ \min_{k \in \Z \cap [-\lceil n T_{\min,a} \rceil, \ n a-\lceil n T_{\min,a} \rceil] \cap [-\A,\A]^\mathsf{C}} \ZZ{n}{\min,a}(k + \UU{n}{\min,a}) < \xi \right) \\
&& \qquad {}+ \PP \left( \ldots, \ \min_{k \in \Z \cap [-\lceil n T_{\min,a} \rceil, \ n a-\lceil n T_{\min,a} \rceil] \cap [-\A,\A]^\mathsf{C}} \ZZ{n}{\min,a}(k + \UU{n}{\min,a}) \geq \xi \right) \\
&& \leq \PP \left( \min_{k \in \Z \cap [-\lceil n T_{\min,a} \rceil, \ n a-\lceil n T_{\min,a} \rceil] \cap [-\A,\A]^\mathsf{C}} \ZZ{n}{\min,a}(k + \UU{n}{\min,a}) < \xi \right) \\
&& \qquad {}+ \PP \left( \min_{k \in \Z \cap [-\lceil n T_{\min,a} \rceil, \ n a-\lceil n T_{\min,a} \rceil] \cap [-\A,\A]} \ZZ{n}{\min,a}(k + \UU{n}{\min,a}) > \xi \right).
\end{eqnarray*}
By the limits~\eqref{ch3:lim A n} and~\eqref{ch3:lim:AGP95 lemma 2} it follows that the $\limsup_{\A \to \infty} \limsup_{n \to \infty}$ of the term~\eqref{ch3:lim:AGP95 lemma 2 A} is
\begin{eqnarray*}
& \leq \PP \left( \min_{k \in \Z} \sigma R (k + U) > \xi \right).
\end{eqnarray*}
Note that this holds for all $\xi \in \R$, so taking $\xi \to \infty$ we obtain the desired limit; that is, that the $\limsup_{\A \to \infty} \limsup_{n \to \infty}$ of the term~\eqref{ch3:lim:AGP95 lemma 2 A} converges to zero.

It follows that for all $s,t,u$ in~$\R$ we have
\begin{eqnarray*}
\lefteqn{ \left\lvert \PP \left( \UU{n}{\min,a} + \argmin_{k \in \Z \cap [-\lceil n T_{\min,a} \rceil, \ n a-\lceil n T_{\min,a} \rceil]} \ZZ{n}{\min,a} (k+\UU{n}{\min,a}) \leq s, \right.\right. } \\
&&\qquad\qquad\ \left. \min_{k \in \Z \cap [-\lceil n T_{\min,a} \rceil, \ n a-\lceil n T_{\min,a} \rceil]} \ZZ{n}{\min,a} (k+\UU{n}{\min,a}) \leq t, \ \UU{n}{\min,a} \leq u \right) \\
&& {} - \PP \left( \UU{n}{\min,a} + \argmin_{k \in \Z \cap [-\lceil n T_{\min,a} \rceil, \ n a-\lceil n T_{\min,a} \rceil] \cap [-\A,\A]} \ZZ{n}{\min,a} (k+\UU{n}{\min,a}) \leq s, \right. \\
&& \qquad\qquad\ \left. \left. \min_{k \in \Z \cap [-\lceil n T_{\min,a} \rceil, \ n a-\lceil n T_{\min,a} \rceil] \cap [-\A,\A]} \ZZ{n}{\min,a} (k+\UU{n}{\min,a}) \leq t, \ \UU{n}{\min,a} \leq u \right) \right\rvert  \\
&& \leq \PP \left( \UU{n}{\min,a} + \argmin_{k \in \Z \cap [-\lceil n T_{\min,a} \rceil, \ n a-\lceil n T_{\min,a} \rceil]} \ZZ{n}{\min,a} (k+\UU{n}{\min,a}) \in [-\A,\A]^\mathsf{C} \right) \\
&& \qquad\qquad\ {} + \PP \left( \min_{k \in \Z \cap [-\lceil n T_{\min,a} \rceil, \ n a-\lceil n T_{\min,a} \rceil] \cap [-\A,\A]^\mathsf{C}} \ZZ{n}{\min,a} (k+\UU{n}{\min,a}) \leq t \right) .
\end{eqnarray*}
Using the limits~\eqref{ch3:lim:AGP95 lemma 2} and~\eqref{ch3:lim:AGP95 lemma 2 A} one obtains that the difference in the previous display goes to zero when taking $\limsup_{n \to \infty}$ and then $\limsup_{\A \to \infty}$.
Together with the limit~\eqref{ch3:lim A n} we conclude the limit in distribution
\begin{align*}
& \lim_{n \to \infty} \ E_{\min,a}^{(n)} \left( \UU{n}{\min,a} ,\ \ZZ{n}{\min,a} , \ T_{\min,a} \right)  = E_{\min,a}^{(\infty)} \left( U ,\ \sigma R , \ T_{\min,a} \right),
\end{align*}
i.e., the limit \eqref{ch3:lim A n 3}, which is what we wanted to prove.

This proves Theorem~\ref{theo2}.
The proof of Theorem~\ref{theo3} is analogous.
\end{proof}

\section*{Acknowledgments}

We gratefully acknowledge support from NSF under grant CMMI-1252878, and also from the program ``Becas de Doctorado en el Extranjero - Becas Chile - CONICYT'' under grant~72110679.
This paper is based on the second author's 2016 PhD thesis at Georgia Institute of Technology.
After defending this thesis, the authors became aware of the recent paper by \cite{ivanovs2016zooming}, which presents a result that generalizes our Theorem~\ref{theo2} to other L\'evy processes.

\appendix\label{ch3:sec:appendix}
\section{Proof of Lemma~1}\label{ch3:sec:appendix:1}


\begin{proof}[Proof of Lemma~\ref{ch3:lemma:decomp}]
We first prove (i).
Using that $\tau_b^n \geq \tau_b$ almost surely, because $B^n(t) = B(\lfloor nt \rfloor / n)$ is piecewise constant as a function of $t$ and $b$ is nondecreasing, we obtain that
\begin{eqnarray*}
\lefteqn{ n \left( \tau_b^n - \tau_b \right) = n \min\{q \in \Z_+/n \cap \left[ \tau_b, \infty \right) : B(q) \geq b(q) \} - n \tau_b } \\
&& = n \min\{q \in \Z_+/n \cap \left[ \tau_b, \infty \right) : B(q) > b(q) \} - n \tau_b, \qquad\qquad\qquad\qquad\qquad\qquad\qquad\quad
\end{eqnarray*}
where the second equality holds almost surely for the Wiener measure.
By $B(\tau_b)=b(\tau_b)$ it follows that
\begin{eqnarray*}
&&= \min \left\{\! k \! \in \! \Z_+ \cap \left[ \lceil n \tau_b \rceil, \infty \right) : \sqrt n \left( B\left(\frac{k}{n}\right) -B\left(\tau_b\right) \right) > \sqrt n \left( b\left(\frac{k}{n}\right)-b\left(\tau_b\right) \right) \right\} \\
&& \qquad\qquad- n\tau_b \\
&&= \min \left\{\! k \! \in \! \Z_+ : \sqrt n \left( B \left( \frac{k + \lceil n\tau_b \rceil}{n} \right) -B(\tau_b) \right) \! > \! \sqrt n \left( b \left( \frac{k + \lceil n\tau_b \rceil}{n} \right) -b(\tau_b) \right) \right\} \\
&& \qquad\qquad+ \lceil n\tau_b \rceil - n\tau_b \\
&&= \min \left\{\! k \! \in \! \Z_+ : \ZZ{n}{\hit,b} \left( k + \UU{n}{\hit,b} \right) > \sqrt n \left( b \left(\tau_b + \frac{k + \UU{n}{\hit,b}}{n} \right) -b(\tau_b) \right) \right\} + \UU{n}{\hit,b},
\end{eqnarray*}
where we used that $\UU{n}{\hit,b} = \lceil n\tau_b \rceil - n\tau_b$ by definition.
Now noting that $\tau_b^n \in \Z_+/n$ we can use the previous identity for $\tau_b^n$ to obtain that
\begin{eqnarray*}
\lefteqn{ \sqrt n \left( B^n (\tau_b^n) - B(\tau_b) \right) = \sqrt n \left( B \left( \tau_b^n \right) - B(\tau_b) \right) } \\
&& = \sqrt n \left( B \left( \tau_b + \frac{n \left( \tau_b^n - \tau_b \right)}{n} \right) - B(\tau_b) \right) \\
&& = \ZZ{n}{\hit,b} \left( n \left( \tau_b^n - \tau_b \right) \right) \\
&& = \ZZ{n}{\hit,b} \! \left( \! \min \{ k \in \Z_+ : \ZZ{n}{\hit,b} \left( k + \UU{n}{\hit,b} \right) \! > \! \sqrt n ( b (\tau_b \! + \! \frac{k + \UU{n}{\hit,b}}{n} ) \! - \! b(\tau_b) ) \} \! + \! \UU{n}{\hit,b} \! \right).
\end{eqnarray*}

We now prove (ii).
Using the definition $\UU{n}{\min,a} = \lceil nT_{\min,a} \rceil - n T_{\min,a}$ it holds
\begin{eqnarray*}
\lefteqn{ n \left( T_{\min,a}^n - T_{\min,a} \right) = n \left( - T_{\min,a} + \argmin_{q \in \Z_+/n \cap [0,a]} B\left( q \right) \right) } \\
&& = - n T_{\min,a} + \argmin_{k \in \Z_+ \cap [0,na]} B\left( \frac{k}{n} \right) \\
&& = - n T_{\min,a} + \argmin_{k \in \Z_+ \cap [-\lceil nT_{\min,a} \rceil, na - \lceil nT_{\min,a} \rceil]} B\left( \frac{k+\lceil nT_{\min,a} \rceil}{n} \right) + \lceil nT_{\min,a} \rceil \\
&& = \UU{n}{\min,a} + \argmin_{k \in \Z_+ \cap [-\lceil nT_{\min,a} \rceil, na - \lceil nT_{\min,a} \rceil]} B\left( T_{\min,a}+ \frac{k+\UU{n}{\min,a}}{n} \right) \\
&& = \UU{n}{\min,a} \! + \! \argmin_{k \in \Z_+ \cap [-\lceil nT_{\min,a} \rceil, na - \lceil nT_{\min,a} \rceil]} \sqrt n \left( B\left( T_{\min,a}+ \frac{k+\UU{n}{\min,a}}{n} \right) \! - \! B\left( T_{\min,a} \right) \right) \\
&& = \UU{n}{\min,a} + \argmin_{k \in \Z_+ \cap [-\lceil nT_{\min,a} \rceil, na - \lceil nT_{\min,a} \rceil]} \ZZ{n}{\min,a}\left( k+\UU{n}{\min,a} \right),
\end{eqnarray*}
and since $T_{\min,a}^n \in \Z_+/n$ then using the previous identity for $T_{\min,a}^n$ we obtain that
\begin{eqnarray*}
\lefteqn{ \sqrt n \left( B^n \left( T_{\min,a}^n \right) - B\left( T_{\min,a}\right) \right) = \sqrt n \left( B \left( T_{\min,a}^n \right) - B\left( T_{\min,a}\right) \right) } \\
&& = \sqrt n \left( B \left( T_{\min,a} + \frac{n \left( T_{\min,a}^n - T_{\min,a} \right)}{n} \right) - B\left( T_{\min,a}\right) \right) \\
&& = \ZZ{n}{\min,a} \left( n \left( T_{\min,a}^n - T_{\min,a} \right) \right) \\
&& = \ZZ{n}{\min,a} \left( \UU{n}{\min,a} +  \argmin_{k \in \Z_+ \cap [-\lceil nT_{\min,a} \rceil, na - \lceil nT_{\min,a} \rceil]} \ZZ{n}{\min,a}\left( k+\UU{n}{\min,a} \right) \right) \\
&& = \min_{k \in \Z_+ \cap [-\lceil nT_{\min,a} \rceil, na - \lceil nT_{\min,a} \rceil]} \ZZ{n}{\min,a}\left( k+\UU{n}{\min,a} \right) .
\end{eqnarray*}

The proof of (iii) is analogous to the two previous ones.
\end{proof}

\section{Proof of Lemma~3}\label{ch3:sec:appendix:3}

\begin{proof}[Proof of Lemma~\ref{ch3:lemma:U convergence}]
It will be sufficient to prove that for all $u \in (0,1)$ and all $t$ such that $\PP(T \leq t)>0$ we have $\PP \left( \left. \lceil n T \rceil -nT \leq u \, \right| \, T \leq t \right) \to u$ as $n \to \infty$.

For that, note that for all $n$ we have
\begin{eqnarray*}
\lefteqn{ \PP\left( \left. \lceil n T \rceil -nT < u \right\rvert T \leq t \right)
= \sum_{k=1}^\infty \PP\left( \left. \left\lceil nT \right\rceil - nT \leq u , \ \left\lceil nT \right\rceil = k \right\rvert T \leq t \right) } \\
&&= \sum_{k=1}^\infty \PP\left( \left. T \in \left[\frac{k-u}{n}, \frac{k}{n} \right] \right\rvert T \leq t \right)
= \sum_{k=1}^\infty \int_{(k-u)/n}^{k/n} f_t (\vvv) \dd \vvv,
\end{eqnarray*}
where $f_t$ is the density with respect to the Lebesgue measure of $T$ conditioned on $\{ T \leq t \}$.
On the other hand, it also holds that
\begin{eqnarray*}
\lefteqn{ \sum_{k=1}^\infty \int_{(k-u)/n}^{k/n} f_t (k/n) \dd \vvv
= \sum_{k=1}^\infty f_t (k/n) \left(\frac{k}{n}-\frac{k-u}{n}\right)
= \sum_{k=1}^\infty f_t (k/n) \frac{u}{n} } \\
&&= \sum_{k=1}^\infty f_t (k/n) u \left(\frac{k}{n}-\frac{k-1}{n}\right)
= u \sum_{k=1}^\infty \int_{(k-1)/n}^{k/n} f_t (k/n) \dd \vvv,
\end{eqnarray*}
and additionally
\begin{eqnarray*}
u \sum_{k=1}^\infty \int_{(k-1)/n}^{k/n} f_t (\vvv) \dd \vvv
= u \int_0^1 f_t (\vvv) \dd \vvv
= u.
\end{eqnarray*}

It follows that since $f_t$ is Riemann-integrable then as $n \to \infty$
$$\left\lvert
\sum_{k=1}^\infty \int_{(k-u)/n}^{k/n} f_t (\vvv) \dd \vvv
-
\sum_{k=1}^\infty \int_{(k-u)/n}^{k/n} f_t (k/n) \dd \vvv
\right\rvert
\to 0$$
and
$$\left\lvert
\sum_{k=1}^\infty \int_{(k-1)/n}^{k/n} f_t (k/n) \dd \vvv
-
\sum_{k=1}^\infty \int_{(k-1)/n}^{k/n} f_t (\vvv) \dd \vvv
\right\rvert
\to 0.$$
Thus, $\PP\left( \left. \lceil n T \rceil -nT < u \right\rvert T \leq t \right) \to u$ as $n \to \infty$, since
\begin{eqnarray*}
\lefteqn{ \left\lvert \PP\left( \left. \lceil n T \rceil -nT < u \right\rvert T \leq t \right) - u \right\rvert } \\
&& \leq \left\lvert \PP\left( \left. \lceil n T \rceil -nT < u \right\rvert T \leq t \right) - \sum_{k=1}^\infty \int_{(k-u)/n}^{k/n} f_t (k/n) \dd \vvv\right\rvert \\
&& \qquad\qquad {}+ \left\lvert u \sum_{k=1}^\infty \int_{(k-1)/n}^{k/n} f_t (k/n) \dd \vvv - u \right\rvert \\
&& = \left\lvert \sum_{k=1}^\infty \int_{(k-u)/n}^{k/n} f_t (\vvv) \dd \vvv - \sum_{k=1}^\infty \int_{(k-u)/n}^{k/n} f_t (k/n) \dd \vvv\right\rvert \\
&& \qquad\qquad {}+ u \left\lvert \sum_{k=1}^\infty \int_{(k-1)/n}^{k/n} f_t (k/n) \dd \vvv - \sum_{k=1}^\infty \int_{(k-1)/n}^{k/n} f_t (\vvv) \dd \vvv \right\rvert,
\end{eqnarray*}
which concludes the proof.
\end{proof}

\bibliographystyle{plainnat}
\bibliography{referencias}

\end{document}